\documentclass[11pt]{amsart}

\usepackage{amscd,amsthm,amsfonts,amssymb,amsmath,esint}
\usepackage{mathrsfs}
\usepackage[colorlinks=true]{hyperref}
\usepackage[all]{xy}
\usepackage{CJK}

    \newcommand{\BA}{{\mathbb {A}}} \newcommand{\BB}{{\mathbb {B}}}
    \newcommand{\BC}{{\mathbb {C}}} 
     \newcommand{\BF}{{\mathbb {F}}}
     \newcommand{\BH}{{\mathbb {H}}}
     
     \newcommand{\BL}{{\mathbb {L}}}
     \newcommand{\BN}{{\mathbb {N}}}
     
    \newcommand{\BQ}{{\mathbb {Q}}} \newcommand{\BR}{{\mathbb {R}}}

     \newcommand{\BZ}{{\mathbb {Z}}}

    \newcommand{\CO}{{\mathcal {O}}} 
     \newcommand{\CR}{{\mathcal {R}}}
    \newcommand{\CS}{{\mathcal {S}}}

    \newcommand{\fa}{{\mathfrak{a}}}

    \newcommand{\fm}{{\mathfrak{m}}} 
     \newcommand{\fp}{{\mathfrak{p}}}
    \newcommand{\fq}{{\mathfrak{q}}}

    \newcommand{\sK}{\mathscr{K}}

\newcommand{\ab}{{\mathrm{ab}}}                     
\newcommand{\ad}{{\mathrm{ad}}}                     

\newcommand{\Aut}{{\mathrm{Aut}}}                   
\newcommand{\Char}{{\mathrm{Char}}}                 
\newcommand{\End}{{\mathrm{End}}}                   
\newcommand{\Gal}{{\mathrm{Gal}}}                   
\newcommand{\GL}{{\mathrm{GL}}}                     
\newcommand{\Hom}{{\mathrm{Hom}}}                   
\newcommand{\Ind}{{\mathrm{Ind}}}                   
\newcommand{\Norm}{{\mathrm{Norm}}}                 
\newcommand{\ord}{{\mathrm{ord}}}                   
\newcommand{\Pic}{\mathrm{Pic}}                     

\newcommand{\tor}{\mathrm{tor}}                     
\newcommand{\Tr}{\mathrm{Tr}}                       
\newcommand{\Vol}{{\mathrm{Vol}}}                   

\renewcommand{\mod}{\, \mathrm{mod}\, }

\font\cyr=wncyr10 \newcommand{\Sha}{\hbox{\cyr X}}

\newcommand{\sk}{\medskip}
\newcommand{\s}{\sk\noindent}
\newcommand{\bs}{\backslash}                        

\newcommand{\ov}{\overline}

\newcommand{\wh}{\widehat}

\newcommand{\ra}{\rightarrow}                           

\newcommand{\pair}[1]{\left\langle {#1}\right \rangle}             
\newcommand{\lrb}[1]{\left(#1\right)}                   
\newcommand{\set}[1]{\left\{#1\right\}}                 

\newcommand{\BetaI}[1]{{\left\{#1 \right\} }}

\newtheorem{thm}{Theorem}[section]

\newtheorem{theorem}[thm]{Theorem}

\newtheorem{coro}[thm]{Corollary}

\newtheorem{lem}[thm]{Lemma}
\newtheorem{lemma}[thm]{Lemma}

\newtheorem{prop}[thm]{Proposition}
\newtheorem{proposition}[thm]{Proposition}

\theoremstyle{definition}

\theoremstyle{remark}

\numberwithin{equation}{subsection}

\def\mat(#1,#2,#3,#4){\begin{pmatrix}#1 & #2 \\ #3 & #4\end{pmatrix}}
\newcommand{\matrixx}[4]{\begin{pmatrix}#1 & #2 \\ #3 & #4\end{pmatrix} }

\begin{document}
\title{Gross-Zagier formula for the $4, 7$ cases of Sylvester's conjecture}
\author{Hongbo Yin}

\begin{abstract}
In \cite{Yin26}, the author constructed some CM points on the elliptic curves $E_{p^i}:y^2=x^3+\frac{p^{2i}}{4}$ for primes $p\equiv 4,7\mod 9$ and $i=1,2$, which give rational points on the curves $x^3+y^3=p^i$. This solves the $4,7$ cases of Sylvester's conjecture. In this paper, we prove the explicit Gross-Zagier formula relating the height of our CM points and the derivative of the $L$-functions of $E_{p^i}$. 
\end{abstract}

\address{School of Mathematics, Shandong University, Jinan, Shandong,  250100,
China}
\email{yhb@sdu.edu.cn}
\maketitle

\section{Introduction}
The Sylvester conjecture states that if $p\equiv 4,7,8\mod 9$, then $p$ should be a sum of two rational cubes, i.e. the equation $x^3+y^3=p$ has non-trivial rational solutions. This is in contrast to the classical result of Sylvester in 1879 that if $p\equiv 2,5\mod 9$, then $p$ cannot be the sum of two rational cubes. In 1993, Elkies claimed a proof of this conjecture for primes $p\equiv 4,7\mod 9$ \cite{Elkies}. But he had never published any details about his proof. In 2006, Dasgupta and Voight proved that if $p\equiv 4,7\mod 9$ and $3\mod p$ is not a cube, then $p$ is a cube sum, see \cite{DV1}\cite{DV17}. In \cite{HSY19}, the authors established the explicit Gross-Zagier formula for Dasgupta-Voight's CM points. This formula explains the reasons why their points work and why they should restrict to the condition that $3\mod p$ is not a cube. Recently, the author gave a proof of the $4,7$ case of Sylvester's conjecture by constructing some non-torsion CM points in \cite{Yin26}. In this paper, we will give the explicit Gross-Zagier formula for our new CM points.

Let us describe our CM point briefly and state the explicit Gross-Zagier formula. In this paper, we assume $i=1,2$. 
For convenience, we will use the elliptic curve $E_{p^i}:y^2=x^3+\frac{p^{2i}}{4}$ which is isogenious to $y^2=x^3-432p^{2i}$ over $\BQ$. Let $K=\BQ(\sqrt{-3})$ be the CM field of $E_{p^i}$ and fix a splitting of $p=\varpi\bar\varpi$ with $\varpi\equiv 1\mod 3$ in $K$. By a theorem of Shimura, we have a modular parametrization 
$$\varphi: X_1(N)\longrightarrow E_{\bar{\varpi}^i}:y^2=x^3+\frac{\bar{\varpi}^{2i}}{4}$$
where $N=9p$ or $27p$ depending on $p^i\equiv 4$ or $7$ modulo $9$. 
We also consider the map
\[\phi: E_{\bar\varpi^i}\longrightarrow E_{p^i}\]
given by
\[(x,y)\mapsto (\sqrt[3]{\varpi^{2i}}x, \varpi^i y).\]

Let $r\in\BZ$ be a solution of $r^2-r+1\equiv 0\mod 3p$ such that $-r\equiv \omega^2\mod\varpi$ where $\omega=-\frac{1}{2}+\frac{\sqrt{-3}}{2}$ is a cubic root of unity. Let $\tau_r=\frac{-1}{3(\omega+r)}$ be the CM point on the upper half plane, then by the result of \cite{Yin26}, $\phi\circ\varphi(\tau_r)\in E_{p^i}(K)$ is a non-torsion point. Let $\Omega_{p^i}$ be the Neron periods of $E_{p^i}$ and $\hat{h}_\BQ$ be the Neron-Tate height with base field $\BQ$ (see \cite[Section 7.1.1]{YZZ} or \cite[Chapter VIII]{Silvermanbook1}). Note that $\hat{h}_\BQ$ is the height used in the BSD conjecture for elliptic curves over $\BQ$. In this paper, we prove the following theorem.
\begin{theorem}\label{main}
We have 
\begin{equation*}
\frac{L'(E_{p^i},1)}{\Omega_{p^i}}=2^{\delta(p^i)}\hat{h}_\BQ(\phi\circ\varphi(\tau_r))
\end{equation*} 
where \[\delta(p^i)=\begin{cases}0,& p^i\equiv 4\mod 9,\\ -1,& p^i\equiv 7\mod 9.\end{cases}\]
\end{theorem}


Our explicit Gross-Zagier formula is very different from previous examples. In fact, we work on the elliptic curves $E_{\bar\varpi}$ which are defined over $K$ not $\BQ$. In this case, the corresponding automorphic representations are not self-dual any more. They correspond to abelian varieties of $\GL_2$-type other than elliptic curves. In previous examples, the modular parametrization is fixed while in our case, the modular parametrization is varied, so the existence of a uniform formula is surprising. The computation of local period integral in this paper also combines all methods developed so far (including the Whittaker model, the Kirillov model and the minimal vector method) to achieve the goal.  

This paper is organized as follows. In section 2, we describe our CM points. In section 3, we prove Theorem \ref{main} assuming some local period integral computations which are delayed to section 4. In section 4, we supply the local period integral computation which is used in section 3.

\subsection*{Acknowledgments} The author would like to thank Xinyi Yuan, Bin Guan, and Jianing Li for many useful discussions. 


\section{CM Point}
Recall the construction of the CM point in this section. For more details please consult \cite{Yin26}.
Let
$$E_{\bar\varpi^i}:y^2=x^3+\frac{\bar\varpi^{2i}}{4}.$$ 
and
$$E_{\varpi^i}:y^2=x^3+\frac{\varpi^{2i}}{4}.$$ 
Let $\psi$, resp. $\psi^c$ be the Hecke character of $E_{\varpi^i}$, resp. $E_{\bar\varpi^i}$, then 
$$\psi(\fq)=\ov{\lrb{\frac{\varpi^i}{\varpi_\fq}}_3}\varpi_\fq$$ with $\fq\nmid 3\varpi$ and $\varpi_\fq\equiv 1\mod 3$. 
Let $\fm$ be the conductor of $\psi$, then 
\begin{equation}
\fm=\begin{cases} (\sqrt{-3})(\varpi),&\text{if}\ \ p^i\equiv 4\mod 9;\\   (\sqrt{-3})^2(\varpi),& \text{if}\ \ p^i\equiv 7\mod 9. \end{cases}
\end{equation}
By \cite{Shimura71} (see also \cite{NM1}), the functions
\begin{equation}\label{def}
f(\tau)=\sum_{\substack{\fa\in I_K(\fm)\\ \fa\ \mathrm{integral}}} \psi(\fa)e^{2\pi i N(\fa)\tau}
\end{equation}
and
\begin{equation}
f^c(\tau)=\sum_{\substack{\fa\in I_K(\bar\fm)\\ \fa\ \mathrm{integral}}} \psi^c(\fa)e^{2\pi i N(\fa)\tau}
\end{equation}
are modular forms on $\Gamma_0(N)$ with Nebentypus character $\xi(d)=\lrb{\frac{-3}{d}}\frac{\psi(d)}{d}$ and $\xi^c(d)=\lrb{\frac{-3}{d}}\frac{\psi^c(d)}{d}=\bar{\xi}(d)$, where 
\begin{equation}\label{conductor}
N=3\Norm(\fm)=\begin{cases} 9p,&\text{if}\ \ p^i\equiv 4\mod 9;\\   27p,& \text{if}\ \ p^i\equiv 7\mod 9. \end{cases}
\end{equation} 
and $I_K(\fm)$ is the group of fractional ideals in $K$ prime to $\fm$.  $f$ is a modular form for the congruence subgroup
\begin{equation}\label{cg}
\Gamma=\left\{\matrixx{a}{b}{c}{d} \in\Gamma_0(N)\mid d\equiv e^3\mod p\ \text{for some}\ e\in\BZ\right\}.
\end{equation}
If we write $f(\tau)=\sum_{n\geq 1}a_n e^{2\pi i n \tau}$ then it can be checked that $f^c(\tau)=\sum_{n\geq 1}\bar{a}_n e^{2\pi i n \tau}$ although $\psi^c\neq\bar\psi$. Also from (\ref{def}), we can see that $a_n\neq 0$ if and only if $n\equiv 1\mod 3$.

Let $I_{f}$ be the annihilator of $f$ in the Hecke algebra which acts on the Jacobian $J_1(N)$. 
By \cite{Shimura71} and \cite{Shimura73} (see also \cite{GL}), $A_{f}=J_1(N)/I_{f}(J_1(N))$ is a two dimensional abelian variety defined over $\BQ$. It splits into $E_{\bar\varpi^i}\times E_{\varpi^i}$ over $K$ by the result in \cite{GL}. The pullback of $\Omega^1(A_{f})$ is the space spanned by $\set{f,f^c}$. But beyond expectation, by \cite[Theorem 1.1(iv)]{GL}, the pullback of the invariant differential of $E_{\bar\varpi^i}$ corresponds to $f$ (rather than $f^c$). By \cite[section 3]{Shimura73} and \cite[Theorem 1]{GL}, $E_{\bar\varpi^i}$ can be parameterized by $X_1(N)$ over $K$ through the integral of $f(z)$, i.e.
\[\varphi: \xymatrix{X_1(N)\ar[r]&J_1(N)\ar[r]& E_{\bar\varpi}}\]
with
\begin{equation}\label{modularpara}
\varphi: t\mapsto z_t=2\pi i\s\int_{i\infty}^t f(z)dz\mapsto \lrb{\wp_L(z_t), \frac{1}{2}\wp'_L(z_t)}.
\end{equation}
Here 
\begin{equation}\label{L}
L=\set{2\pi i \int_{i\infty}^{\gamma \infty}f(\tau) d\tau\mid \gamma\in\Gamma_1(N)}
\end{equation}
is the period lattice of $f$ and in the parametrization we use the fact that the Manin-Stevens constant of $E_{\bar\varpi^i}$ is a unit which is proved in \cite{Yin26}. We can similarly define $\varphi^c,\phi^c$ for $E_{\varpi}$.

Let
\[M=\matrixx{0}{-1}{3}{3r},\]
where $r\in \BZ$ such that $r^2-r+1\equiv 0\mod 3p$ and $-r\equiv \omega^2\mod\varpi$. We define the CM point to be $$\tau_r=M\omega=\frac{-1}{3\omega+3r}.$$ Then we have the normalized embedding $\iota_1$ of $K$ into the $2\times 2$ matrix algebra $M_2(\BQ)$ in the sense of \cite{Shimurabook} such that

\begin{equation}\label{embedding}
\iota_1(\omega)=M\matrixx{-1}{-1}{1}{0}M^{-1}=\matrixx{-r}{-1/3}{3(r^2-r+1)}{r-1},
\end{equation}
then
\[\iota_1(\omega^2)=\matrixx{r-1}{1/3}{-3(r^2-r+1)}{-r}.\]
Let $\sigma: \wh{K}^\times\ra \Gal(K^\ab/K)$ be the Artin reciprocity law and we denote by $\sigma_t$ the image of $t\in\wh{K}^\times$. The following results are proved in \cite{Yin26}. 
\begin{prop}\label{LCF}
Let the notation be as above, we have
\begin{itemize}
\item[1.] The field $K(\sqrt[3]{\varpi})$ is contained in $H_{3\varpi}$.
\item[2.] We have 
\[\left(\sqrt[3]{\varpi}\right)^{\sigma_{\omega_\fp}-1}=\left\{\begin{aligned} {\omega},&\quad p\equiv 4\mod 9, \\{\omega^2},&\quad p\equiv 7\mod 9.  \end{aligned}\right.\]
So, $\Gal(K(\sqrt[3]{\varpi})/K)=\pair{\sigma_{\omega_\fp}}$.
\end{itemize}
\end{prop}
\begin{theorem}\label{non-torsion}
Let $r$ be an integer such that $r^2-r+1\equiv 0\mod 3p$. If $-r\equiv \omega^2\mod\varpi$, 
then $\varphi(\tau_r)$ is defined over $K(\sqrt[3]{\varpi})$ and 
\[\sigma_{\omega_\fp}(\varphi(\tau_r))=\begin{cases}[\omega^2]\varphi(\tau_r),& \text{if}\ \ N=9p;\\ [\omega]\varphi(\tau_r),& \text{if}\ \ N=27p, \end{cases}\]
Moreover, $\phi\circ\varphi(\tau_r)\in E_{p^i}(K)$ is non-torsion.  
\end{theorem}

\section{Explicit Gross-Zagier formula}\label{EGZ}
In this section, we will give the explicit Gross-Zagier formula for our CM points which does not satisfy the Heegner hypothesis using the general formula in \cite{YZZ} and its explicit version in \cite{CST14}. This requires computing every factor in \cite[Theorem 1.6]{CST14} very clearly in our cases which is done in subsections \ref{H}-\ref{A} and section \ref{L}.

By \cite[P270]{NM1}, $\End^0_{\BQ}(A_f)=K$, so $A_f$ is of $\GL_2$-type, see also \cite[Proposition 15.1.5]{Ribet-Stein}. For simplicity's sake, we will use $A$ instead of $A_f$ from now on. Since $A$ is isogenous to $E_{\bar{\varpi}^i}\times E_{\varpi^i}$ over $K$ (see \cite[Theorem 1.1]{GL} also \cite[Lemma 4.3]{NM1}), we work directly with the abelian variety $A_{/\BQ}=\text{Res}_{K/\BQ}E_{\bar{\varpi}^i}=\text{Res}_{K/\BQ}E_{\bar{\varpi}^i}$. By \cite[Proposition 1]{CN}, $A$ is self-dual. Then we will identify $A$ with $A^\vee$ freely. Let $\pi$ be the automorphic representation of $\GL_2(\BQ)$ corresponding to the theta series $f$, then $\{\pi,\bar\pi\}$ form an orbit  under the action of $\Aut(\BC)$. Then by \cite[Theorem 3.4(2)]{YZZ}, there is a unique automorphic representation $\pi_f$ over $\BQ$ such that $\pi=\pi_f\otimes_{(K,\iota_1)}\BC$, $\bar\pi=\pi_f\otimes_{(K,\iota_2)}\BC$ where $\iota_1,\iota_2$ are the two different embeddings of $K$ into $\BC$ and $\pi_f\otimes_{\BQ}\BC=\pi\oplus \bar\pi$.  

Recall that $f$ is constructed from the Hecke character $\psi$ and $\psi(\fq)=\ov{\lrb{\frac{\varpi^i}{\fq}}_3}\varpi_\fq$ where $\varpi_\fq$ is a generator of the prime ideal $\fq$ such that $\varpi_\fq\equiv 1\mod 3$ for $\fq\nmid 3\varpi$. Let $\Theta$ be the unitarization of $\psi$, that is $\Theta(\fq)=\psi(\fq)(N\fq)^{-\frac{1}{2}}$. It is well known that $\pi$ is the Weil representation $\pi\lrb{\Ind_{W_K}^{W_\BQ}\Theta}$ described in \cite[P124]{Labesse79} with central character $w_\pi=\lrb{\frac{-3}{\cdot}}\Theta|_{\BA^\times}$ , where $\lrb{\frac{-3}{\cdot}}$ is the adelization of the Legendre symbol. Let $\chi$ be the character of $\Gal(\bar{K}/K)$ such that $\chi(\sigma)=\sigma(\sqrt[3]{\varpi^i})/(\sqrt[3]{\varpi^i})$ and also view it as a Hecke character or Grossencharacter of $K$ through the class field theory. In particular $\chi(\fq)=\lrb{\frac{\varpi^i}{\fq}}_3$. It is easy to check that $w_{\pi}(q)=\overline{\lrb{\frac{\varpi^i}{q}}_3}=\chi(q)^{-1}$ for $q\nmid 3p$, so in fact  $w_\pi=\chi^{-1}\mid_\BQ$. The base change of $\pi$ to $K$ is the principal series $\pi_K=\pi(\Theta,\Theta^c)$ where $\Theta^c(z)=\Theta(\bar z)$, see \cite{Labesse79} or \cite[Appendix E]{JMbook}. Note that in general $\Theta^c$ is different from $\bar{\Theta}$  where $\bar\Theta(z)=\ov{\Theta(z)}$, so $\pi_K$ is not necessarily the same as the principal series $\pi(\Theta,\bar{\Theta})$ which is the automorphic representation corresponding to $E_{\varpi^i}$.  Let $L(s,\pi_K,\chi)$ be the automorphic $L$-function of $\pi_K\otimes\chi$. We have the following relationship between the automorphic $L$-functions and the Hasse-Weil $L$-functions.

\begin{proposition}\label{Lfunction}
$$L\lrb{s-\frac{1}{2},\pi_K,\chi}=L(E_1,s)L(E_{p^i},s)$$
\end{proposition}
\begin{proof}
$L\lrb{s-\frac{1}{2},\pi_K,\chi}=L(s-\frac{1}{2},\chi\Theta)L(s-\frac{1}{2},\chi\Theta^c)$. 
By definition, $\chi$ is the cubic residue symbol such that $\chi(\fq)=\lrb{\frac{\varpi^i}{\fq}}_3$
for $\fq\nmid 3\varpi$. So $\chi\Theta(\fq)=\varpi_\fq(N\fq)^{-\frac{1}{2}}$ for any $\fq\nmid 3$.  
As a result,
\begin{eqnarray*}
L\lrb{s-\frac{1}{2},\chi\Theta}=\prod_{\fq\nmid 3}\lrb{1-q^{-s}\varpi_\fq}^{-1}=L(s,E_1).
\end{eqnarray*}

Since 
$${\lrb{\frac{\varpi^i}{\fq}}_3}=\ov{\lrb{\frac{\bar{\varpi}^i}{\bar\fq}}_3},$$ 
we get
\[\chi\Theta^c(\fq)=\lrb{\frac{\varpi^i}{\fq}}_3\ov{\lrb{\frac{\varpi^i}{\bar\fq}}_3}\bar\varpi_\fq(N\fq)^{-\frac{1}{2}}=\ov{\lrb{\frac{p^i}{\bar\fq}}_3}\bar\varpi_\fq(N\fq)^{-\frac{1}{2}}\]
for $\fq\nmid 3p$. So
\begin{eqnarray*}
L\lrb{s-\frac{1}{2},\chi\Theta^c}=\prod_{\fq\nmid 3p}\lrb{1-q^{-s}\ov{\lrb{\frac{p^i}{\bar\fq}}_3}\bar{\varpi}_\fq}^{-1}=L(s,E_{p^i}).
\end{eqnarray*}
\end{proof}

For any place $v$ of $\BQ$, let $\epsilon(1/2,\pi_v\times \pi_{\chi_v})$ be the local epsilon factor of the Rankin-Selberg  $L$-function of $\pi$ and $\pi_\chi$ where $\pi_\chi$ is the automorphic representation corresponding to the theta series of $\chi$. For a quaternion algebra $\BB$ over $\BA$, $\epsilon_v(\BB)=1$ if $\BB_v$ is split and $\epsilon_v(\BB)=-1$ otherwise. The quadratic character associated with the extension $K/\BQ$ will be denoted by $\eta$. Then we have the following Tunnell-Saito condition.
\begin{prop}\label{Tun-Saito}
The incoherent quaternion algebra $\BB$ over $\BA_\BQ$, which satisfies
$$\epsilon(1/2,\pi_v\times \pi_{\chi_v})=\chi_{v}(-1)\eta_v(-1)\epsilon_v(\BB)$$
for all places $v$ of $\BQ$, is only ramified at the infinity place.
\end{prop}
\begin{proof} Since $\chi$ is a cubic character, $\chi_v(-1)=1$ for any $v$. Also, $\eta_v(-1)=-1$ if and only if $v$ is ramified in $K$. Since $\pi$ is unramified at finite places $v\nmid 3p$, $\chi$ is unramified at finite places $v\nmid 3p$ and $p$ is split in $K$, by \cite[Proposition 6.3]{Gross88}(see also \cite[Lemma 3.1]{CST14}) we get $\epsilon(1/2,\pi_v\times\pi_{\chi_v})=+1$ for all  finite $v\neq 3$. Again by \cite[Proposition 6.5]{Gross88}, we also know that $\epsilon(1/2,\pi_\infty\times\pi_{\chi_\infty})=-1$. Since $\epsilon(1/2,\pi\times\pi_{\chi})=-1$, we see that $\epsilon(1/2,\pi_3\times\pi_{\chi_3})=+1$.  Hence $\BB$ is only ramified at the infinity place.
\end{proof}

Let $\mathbb{B}_f^{\times}=\mathrm{GL}_2\left(\mathbb{A}_f\right)$ be the finite part of $\mathbb{B}^{\times}$. For any open compact subgroup $U \subset \mathbb{B}_f^{\times}$, the Shimura curve $X_U$ associated with $\mathbb{B}$ of level $U$ is the usual modular curve with complex uniformization
$$
X_U(\mathbb{C})=\mathrm{GL}_2(\mathbb{Q})^{+} \backslash\left(\mathcal{H} \sqcup \mathbb{P}^1(\mathbb{Q})\right) \times \mathrm{GL}_2\left(\mathbb{A}_f\right) / U.
$$
We set
\[U_0(N)=\set{\matrixx{a}{b}{c}{d}\in\GL_2(\widehat\BZ)\mid c\equiv 0 \mod N\widehat{\BZ}}, \]
and
\[U_1(N)=\set{\matrixx{a}{b}{c}{d}\in U_0(N)\mid d\equiv 1 \mod N\widehat{\BZ}}.\]
Then $X_{U_0(N)}=X_{\Gamma_0(N)}$ and $X_{U_1(N)}=X_{\Gamma_1(N)}$.
In \cite{YZZ},  the Jacquet-Langlands correspondence of $\pi_f$ to $\BB$ is realized as 
$$\pi_A:=\varinjlim_{U}\Hom^0_{\xi_U}(X_U, A)$$ 
where $\Hom^0_{\xi_{U}}(X_U,A)$ denotes the morphisms in $\Hom_\BQ(X_U,A)\otimes_\BZ\BQ$ using the Hodge class $\xi_U$ as a base point. Then the $L$-function $L(s, \pi_A)$ is defined to be $L(s,\pi_f)$.
For the definition of the $L$-functions associated to the $\GL_2$-type abelian variety $A$, please refer to \cite[P7,P69]{YZZ} and \cite[P2530]{CST14}. In particular, it can be constructed from the $\ell$-adic representation of $A$ over $\BQ$ twisted by the induced representation of $\chi$ from $\Gal(\bar K/K)$ to $\Gal(\bar\BQ/\BQ)$. It can also be constructed from the $\ell$-adic representation of $A_K$ over $K$ twisted by $\chi$. In particular, it is valued in $K\otimes_{\BQ}\BC$ and can be viewed as vector-valued indexed by the embeddings of $\iota: K\hookrightarrow \BC$. And the Gross-Zagier formula will be understood as an equality of every component of the vector.  \cite[Theorem 3.8]{YZZ} says $L(s, A)=(L(s, \pi), L(s,\bar\pi))$. By Proposition \ref{Lfunction}, we have
\begin{eqnarray*}
L\lrb{s,A,\chi}&=&\lrb{L\lrb{s-\frac{1}{2},\pi_K,\chi},L\lrb{s-\frac{1}{2},\bar\pi_K,\bar\chi}}\\
                       &=&(L(E_1,s)L(E_{p^i},s), L(E_1,s)L(E_{p^i},s)).
\end{eqnarray*}

By \cite[Theorem 1.6]{CST14}, we have
\begin{equation}\label{G-Z}
L^{(p)'}\lrb{\frac{1}{2},\pi_K,\chi}=\frac{8\pi^2(f,f)_{\Gamma_0(N)}}{2\sqrt{3}p}\cdot\frac{\left\langle P^0_{\chi}(\varphi), P^0_{\chi^{-1}}(\varphi)\right\rangle_{K,K}}{(\varphi,\varphi)_{\CR^\times}}\cdot \frac{\beta^0(\varphi'_3,\varphi'_3)}{\beta^0(\varphi_3,\varphi_3)}\frac{\beta^0(\varphi'_p,\varphi'_p)}{\beta^0(\varphi_p,\varphi_p)},
\end{equation}
where $L^{(p)}(s,\pi_K,\chi)$ is the partial $L$-function with $p$-factor missing, $(\cdot,\cdot)_{\Gamma_0(N)}$ is the Petersson norm, $P^0_{\chi}(\varphi)$ is the Heegner cycle, $\langle\cdot,\cdot\rangle_{K,K}$ is a pairing from $A(\bar{K})_\BQ\otimes_K A^\vee(\bar{K})_\BQ$ to $\BC$ such that $\langle\cdot,\cdot\rangle_{K}=\rm{Tr}_{\BC/\BR}\langle\cdot,\cdot\rangle_{K,K}$ is the map from $A(\bar{K})_\BQ\otimes_K A^\vee(\bar{K})_\BQ$ to $\BR$ given by  the Neron-Tate height $\hat{h}_K$ over the base field $K$ (\cite[Section 1.2.4 and 7.1.1]{YZZ}), 
$\CR$ is the admissible order of $M_2(\wh{\BZ})$ for the pair $(\pi_A,\chi)$ (\cite[Definition 1.3]{CST14}) and $(\cdot,\cdot)_{\CR^\times}$ is the pairing on $\pi_A\times \pi_{A^\vee}$ defined as in \cite[Page 789]{CST17}. Note that in \eqref{G-Z} we identify $A^\vee$ with $A$ since $A$ is self-dual. Finally, the local period integral $\beta^0(\cdot,\cdot)$ is defined and computed in Section \ref{L}.

\subsection{Heegner cycles} \label{H}
In our case, the conductor of the character $\chi$ is $3\varpi$ by Proposition \ref{LCF}. The Heegner cycle in \cite{YZZ} and \cite{CST14} is 

\begin{equation*}P^0_{\chi}(\varphi) =\frac{\sharp \Pic(\CO_{p})}{\Vol(K^\times\bs\widehat{K}^\times)}
\int_{K^\times\backslash \widehat{K}^\times}\varphi(\tau_r)^{\sigma_t}\otimes\chi(t)dt, 
\end{equation*}
here $\CO_p$ is the order of conductor $p$ of $\CO_K$. Unfolding the integral we have
\begin{eqnarray*}
P^0_{\chi}(\varphi)
&=& \frac{\# \Pic(\CO_p)}{\# K^\times\backslash \widehat{K}^\times/U_{3\varpi}}\sum_{t\in K^\times\backslash \widehat{K}^\times/U_{3\varpi}}\varphi(\tau_r)^{\sigma_t}\otimes\chi(t)\\
&=& \frac{\# \Pic(\CO_p)\#\Gal(H_{3\varpi}/K(\sqrt[3]{\varpi}))}{\# K^\times\backslash \widehat{K}^\times/U_{3\varpi}}\sum_{\sigma\in \Gal(K(\sqrt[3]{\varpi})/K)}\varphi(\tau_r)^{\sigma}\otimes\chi(\sigma)\\
&=&\frac{p-1}{9}\sum_{\sigma\in \Gal(K(\sqrt[3]{\varpi})/K)}\varphi(\tau_r)^{\sigma}\otimes\chi(\sigma).
\end{eqnarray*}
Similarly,

\begin{eqnarray*}P^0_{\chi^{-1}}(\varphi)=\frac{p-1}{9}\sum_{\sigma\in \Gal(K(\sqrt[3]{\varpi})/K)}\varphi(\tau_r)^{\sigma}\otimes \chi^{-1}(\sigma).
\end{eqnarray*}
Here the sum in the last expression just means a linear combination (not the complex multiplication and the addition on the group of rational points).

\begin{eqnarray*}&&\langle P^0_{\chi}(\varphi),P^0_{\chi^{-1}}(\varphi)\rangle_{K,K}\\&=&\frac{(p-1)^2}{81}
\langle\sum_{\sigma\in \Gal(K(\sqrt[3]{\varpi})/K)}\varphi(\tau_r)^{\sigma}\otimes\chi(\sigma),\sum_{\sigma\in \Gal(K(\sqrt[3]{\varpi})/K)}\varphi(\tau_r)^{\sigma}\otimes\chi^{-1}(\sigma)\rangle_{K,K}\\
&=&\frac{(p-1)^2}{27}\langle \varphi(\tau_r),\sum_{\sigma\in\Gal(K(\sqrt[3]{\varpi})/K)}\varphi(\tau_r)^\sigma\otimes\chi(\sigma)\rangle_{K,K}\\
&=&\frac{(p-1)^2}{27}\left(\langle \varphi(\tau_r),\varphi(\tau_r)\rangle_{K,K}-\left\langle \varphi(\tau_r),\varphi(\tau_r)^{\sigma'}\right\rangle_{K,K}\right),
\end{eqnarray*}
where $\sigma'$ is a generator of $\Gal(K(\sqrt[3]{\varpi})/K)$. In the last equality we use the fact that $\langle \varphi(\tau_r),\varphi(\tau_r)^{\sigma'}\rangle_{K,K}=\langle \varphi(\tau_r),\varphi(\tau_r)^{\sigma'^2}\rangle_{K,K}$ since $\langle,\rangle_{K,K}$ is symmetric and Galois invariant. We can assume $\varphi(\tau_r)^{\sigma'}=[\omega]\varphi(\tau_r)$, then
$$\left\langle \varphi(\tau_r),\varphi(\tau_r)^{\sigma'}\right\rangle_{K,K}=\frac{1}{2}\left(\widehat{h}_K([1+\omega]\varphi(\tau_r))
-\widehat{h}_K([\omega]\varphi(\tau_r))-\widehat{h}_K(\varphi(\tau_r))\right).$$
Since $|1+\omega|=|\omega|=1$, by definition, $\widehat{h}_K([1+\omega]\varphi(\tau_r))=\widehat{h}_K([\omega]\varphi(\tau_r))=\widehat{h}_K(\varphi(\tau_r))$. Then
$$
\left\langle \varphi(\tau_r),\varphi(\tau_r)^{\sigma'}\right\rangle_{K,K}=-\frac{1}{2}\widehat{h}_K(\varphi(\tau_r)),$$
and hence
\begin{equation}\label{HeightPair}
\left\langle P^0_{\chi}(\varphi),P^0_{\chi^{-1}}(\varphi)\right\rangle_{K,K}=\frac{(p-1)^2}{2\cdot 9}\widehat{h}_K(\varphi(\tau_r))=\frac{(p-1)^2}{9}\widehat{h}_\BQ(\phi\circ\varphi(\tau_r)),
\end{equation}
here we use the fact that the Neron-Tate height is invariant under under isomorphisms given by scaling of coordinates.
\subsection{Petersson norm of $f$}\label{P}
Let 
\[(f,f)_{\Gamma_1(N)}=\int\int_{\Gamma_1(N)\bs\BH}|f(z)|^2 dz\]
and
\[(f,f)_{\Gamma_0(N)}=\int\int_{\Gamma_0(N)\bs\BH}|f(z)|^2 dz\]
be the Petersson norms of $f$ concerning different congruence subgroups. By \cite[Page 374]{Zagier85}, 
$$(f,f)_{\Gamma_1(N)}=\frac{\deg \varphi}{4\pi^2 c\bar c}\cdot\Vol(E_{\bar{\varpi}^i}),$$ 
here $\varphi:X_1(N)_{/K}\ra E_{\bar\varpi^i}$ is the modular parametrization and $c$ is the Manin-Stevens constant which is proved to be a unit in \cite[Section 6]{Yin26}. Let $\Omega_1$ and $\Omega_{p^i}$ be the real periods of $E_1$ and $E_{p^i}$ associated to the invariant differential $\frac{dx}{2y}$. These are the minimal periods that appear in the BSD conjecture by the minimal equations of $E_1$ and $E_{p^i}$ described in \cite[Lemma 7]{Yin}.  Using SageMath we can check that $[\Omega_1,\Omega_1\frac{1}{\sqrt{-3}}(\frac{-1}{2}+\frac{\sqrt{-3}}{2})]$ is a basis of the period lattice of $E_1$. Since $E_{p^i}$ is isomorphic to $E_1$ over $\BR$ (in particular their invariant differential is different by a factor $\sqrt[3]{p^i}$), $[\Omega_{p^i},\Omega_{p^i}\frac{1}{\sqrt{-3}}(\frac{-1}{2}+\frac{\sqrt{-3}}{2})]$ is also a basis of the period lattice of $E_{p^i}$. So $\Vol(E_1)=\frac{1}{2\sqrt{3}}\Omega_1^2$ and $\Vol(E_{p^i})=\frac{1}{2\sqrt{3}}\Omega_{p^i}^2$. Since $\frac{dx}{2\sqrt{x^3+\frac{\bar{\varpi}^{2i}}{4}}}$ is the invariant differential on $E_{\bar\varpi^i}$ that corresponds to $dz$ on $\BC/\Lambda$ where $\Lambda$ is the period lattice of $E_{\bar\varpi^i}$ (see \cite[P171]{Silvermanbook1}), we have

\[\Vol(E_{\bar\varpi^i})=\frac{1}{2\sqrt{-1}}\int_{\BC/\Lambda}dz\wedge d\bar z=\frac{1}{2\sqrt{-1}}\int_{E_{\bar\varpi}(\BC)}\frac{dx}{2\sqrt{x^3+\frac{\bar{\varpi}^{2i}}{4}}}\wedge \frac{d\bar x}{2\sqrt{{\bar x}^3+\frac{{\varpi}^{2i}}{4}}}.\]
Changing variable $x\mapsto \sqrt[3]{{\bar\varpi}^{2i}}x$, then $\bar x\mapsto \sqrt[3]{{\varpi}^{2i}}\bar x$ and the integral domain changes to $E_1(\BC)$, so 
\[\Vol(E_{\bar\varpi^i})=\frac{1}{2\sqrt{-1}}\frac{1}{\sqrt[3]{p^i}}\int_{E_{1}(\BC)}\frac{dx}{2\sqrt{x^3+\frac{1}{4}}}\wedge \frac{d\bar x}{2\sqrt{{\bar x}^3+\frac{1}{4}}}=\frac{1}{\sqrt[3]{p^i}}\Vol(E_1).\]
Similarly $\Vol(E_{p^i})=\frac{1}{\sqrt[3]{p^{2i}}}\Vol(E_1)$. As a result, 
$$\Vol(E_{\bar\varpi^i})^2=\Vol(E_{p^i})\Vol(E_1)=\frac{1}{12}(\Omega_1\Omega_{p^i})^2$$ 
and $\Vol(E_{\bar\varpi^i})=\frac{1}{2\sqrt{3}}\Omega_1\Omega_{p^i}$. So, finally, we have
\begin{equation}\label{Petter}
(f,f)_{\Gamma_0(N)}=\frac{1}{\vartheta(N)}(f,f)_{\Gamma_1(N)}=\frac{\deg\varphi}{8\sqrt{3}\pi^2\vartheta(N)}\Omega_1\Omega_{p^i},
\end{equation}
where $\vartheta$ is the Euler function.

\subsection{Pairing of $\pi_A$}\label{Pa}
Let $A^\vee$ be the dual of $A$, there is a perfect pairing
\[(\cdot,\cdot)_U: \pi_A\times\pi_{A^\vee}\ra \End_\BQ(A)\cong K\]
given by
\[(v_1,v_2)_U\mapsto v_1\circ v_2^\vee.\]
where $v_1\in\Hom(J_U,A),v_2\in\Hom(J_U,A^\vee)$ and $v_2^\vee$ is the dual of $v_2$ composed with the canonical isomorphism $J_U^\vee\cong J_U$. See \cite[3.2.4]{YZZ} about more information on this pairing.
The pairing of new forms appears in the explicit Gross-Zagier formulae, so we need to compute it explicitly.

Since the functor of Weil restriction is right adjoint to the functor of base change, we have $\Hom_\BQ(J_1(N)_{/K}, A)=\Hom_K(J_1(N), E_{\bar\varpi})$. Let $\varphi_\BQ$ be the map corresponding to $\varphi$, i.e. $\varphi_\BQ$ is the Weil restriction of $\varphi$, then $\varphi_\BQ^\vee\in\Hom_\BQ(A,J_1(N))$ is also the Weil restriction of $\varphi^\vee\in\Hom_K(E_{\bar\varpi},J_1(N)_{/K})$. As a result, $\varphi_\BQ\circ \varphi_\BQ^\vee$ is also the Weil restriction of $\varphi\circ \varphi^\vee$. It is well-known that $\varphi\circ \varphi^\vee$ is the isogeny $[\deg \varphi]$, so $\varphi_\BQ\circ \varphi_{\BQ}^\vee$ is also the isogeny $[\deg \varphi]$. So we have:
\begin{equation}\label{newformpair} 
(\varphi_\BQ,\varphi_\BQ)_{U_1(N)}=\deg \varphi.
\end{equation}

\subsection{Admissible orders}\label{A}
The definition of admissible order in \cite{CST14} is very technical, so we will not write it here, please refer to the original paper.
Let $\CR$ be the admissible order for $(\pi,\chi)$ defined in \cite[Definition 1.3]{CST14}, then $\CR^\times$ is different from $U_1(N)$ only at places $3$ and $p$. In this subsection, we give the admissible orders at places $p$ and $3$ explicitly and compute their volumes with respect to the Tamagawa number measure which is needed to compute the pairing $(\varphi_\BQ,\varphi_\BQ)_{\CR^\times}$.

\subsubsection{Admissible orders at place $\varpi$}

Let $K_p=K\otimes\BZ_p$ and $\CO_{K_p}=\CO_K\otimes\BZ_p$.	Note that 
$$\iota_1(\CO_{K}\otimes\BZ_p)=\matrixx{\BZ_p}{\BZ_p}{p\BZ_p}{\BZ_p}.$$ Let 
$$R'=\matrixx{\BZ_p}{p\BZ_p}{p^{-1}\BZ_p}{\BZ_p}\ \text{and}\ R''=\matrixx{\BZ_p}{\BZ_p}{\BZ_p}{\BZ_p}$$ 
be the maximal orders such that $R'\cap \iota_1(K\otimes\BZ_p)=\CO_p$ and $R''\cap \iota_1(K\otimes\BZ_p)=\CO_{K_p}$. Then 
$$\CR_p=R'\cap R''=\matrixx{\BZ_p}{p\BZ_p}{\BZ_p}{\BZ_p}$$ is the admissible order for $(\pi_p,\chi_p)$ by the definition in \cite[Page 2530]{CST14}. The maximal ideals of $\CR_p$ are 
$$\matrixx{p\BZ_p}{p\BZ_p}{\BZ_p}{\BZ_p}\ \text{and}\ \matrixx{\BZ_p}{p\BZ_p}{\BZ_p}{p\BZ_p}.$$ 
As a result, the radical of $\CR_p$ is 
$$\text{Rad}(\CR_p)=\matrixx{p\BZ_p}{p\BZ_p}{\BZ_p}{p\BZ_p}$$ 
and $\CR_p/\text{Rad}(\CR_p)=\BF_p^2$. By \cite[Lemma 3.5]{CST14}, 
$$\Vol(\CR_p^\times)=\Vol(U_0(N)_p)=\vartheta(p)\Vol(U_1(N)_p),$$ the volume here is with respect to the Haar measure such that $\Vol(\GL_2(\BQ_p))=L(2,1_{\BQ_p})^{-1}$. 

\subsubsection{Admissible orders at place $3$}
Let $\CR_3$ be the $\BZ_3$-order of $M_2(\BZ_3)$ with discriminant $N/p$ satisfying
$\CR_3\cap K_3=\CO_{K_3}$ which is unique by \cite[Lemma 3.4]{CST14}. By \cite[Page 1159]{Gross88}, we can take $\CR_3=\CO_{K_3}+I\bar{\CR}$ where
$\bar{\CR}$ is a maximal order of $M_2(\BZ_3)$ containing $\CR_3$ and $I$ is an ideal of $\CO_{K_3}$ such that $\ord_3(N/p)=d(\bar{\CR})+{\text{length}_\BZ(\CO_{K_3}/I)}$, Here $d(\bar{\CR})$ is the exponent of the discriminant of $\bar{\CR}$.
It is easy to see that $\text{Rad}(\CR_3)=\sqrt{-3}\CO_{K_3}+I\bar{\CR}$ and $\CR_3/(\text{Rad}(\CR_3))=\BF_3$. By the \cite[Lemma 3.5]{CST14}, $$\Vol(\CR_3^\times)=\frac{3}{2}\Vol(U_0(N)_3)=\frac{3\vartheta(N/p)}{2}\Vol(U_1(N)_3).$$

Finally, we have
$$\frac{\Vol(\CR^\times)}{\Vol(U_1(N))}=\frac{\Vol(\CR^\times_3)}{\Vol(U_1(N)_3)}\cdot\frac{\Vol(\CR^\times_p)}{\Vol(U_1(N)_p)}=\frac{3\vartheta(N)}{2}.$$ 
So 
\begin{equation}\label{DualPair}
(\varphi_\BQ,\varphi_\BQ)_{\CR^\times}=\frac{\Vol(X_{\CR^\times})}{\Vol(X_{U_1(N)})}\deg \varphi=\frac{\Vol(U_1(N))}{\Vol(\CR^\times)}\deg \varphi=\frac{2}{3\vartheta(N)}\deg \varphi.
\end{equation}

\subsection{The explicit formula}By (\ref{G-Z}), (\ref{HeightPair}), (\ref{Petter}), (\ref{newformpair}), (\ref{DualPair}), (\ref{W1}) and Proposition \ref{TestingForNew}, \ref{Prop:TestingForNew}, We get
\begin{equation}
\frac{L(E_1,1)}{\Omega_1}\frac{L'(E_{p^i},1)}{\Omega_{p^i}}=\frac{2^{\delta(p^i)}}{9}\hat{h}_\BQ(\varphi(\tau))
\end{equation}
with
\[\delta(p^i)=\begin{cases}0,& p^i\equiv 4\mod 9,\\ -1,& p^i\equiv 7\mod 9.\end{cases}\]
The value $\frac{L(E_1,1)}{\Omega_1}$ can be computed explicitly  which is $\frac{1}{9}$. So we get 

\begin{equation}\label{GZ}
\frac{L'(E_{p^i},1)}{\Omega_{p^i}}=2^{\delta(p^i)}\hat{h}_\BQ(\varphi(\tau)).
\end{equation}
This proves Theorem \ref{main}. 

By \cite[section 5]{Stephens}, the local Tamagawa numbers of $E_{p^i}$ are given as follows:
\[c_v(E_{p^i})=\begin{cases}1,& v\mid 3\ \text{and}\ p^i\equiv 4\mod 9,\\ 2,& v\mid 3\ \text{and}\ p^i\equiv 7\mod 9,
\\ 3,& v\mid p, \\ 1,& v\nmid 3p. \end{cases}\]
Since $|E_{p^i}(\BQ)_\tor|=3$, the BSD conjecture (\cite[Section C.16]{Silvermanbook1}, but the Neron-Tate height in \cite{YZZ} is twice the height in \cite{Silvermanbook1}) predicts that
\begin{equation}\label{BSD}
\frac{L'(E_{p^i},1)}{\Omega_{p^i}}=\frac{2^{-\delta(p^i)}\Sha{(E_{p^i})}\hat{h}_\BQ(P)}{3},
\end{equation}
where $P$ is the generator of $E_{p^i}(\BQ)$. Then (\ref{GZ}) together with (\ref{BSD}) predict that $[\sqrt{-3}]\varphi(\tau_r)$ is usually the generator or twice the generator (depending on $N=9p$ or $27p$) of $E_{p^i}(\BQ)$ if $\Sha{(E_{p^i})}=1$. 

\section{Local Waldspurger integral}\label{L}
In this last section, we compute the local integrals which appear in the Gross-Zagier formula.

Recall that $\pi$ is the automorphic representation of $\GL_2(\BQ)$ corresponding to the theta series $f$ and $\chi:\Gal(\bar{K}/K)\ra \CO_K^\times$ is the character given by $\chi(\sigma)=(\sqrt[3]{\varpi})^{\sigma-1}$. We also view $\chi$ as a Hecke character on $\BA_K^\times$ by the Artin map. For $q=p$ or $3$, $\pi_q$ and $\chi_q$ will be the $q$ component of $\pi$ and $\chi$. Assume $v_q$ is a vector in $\pi_q$ and $(\cdot,\cdot)$ is an invariant Hermitian form on $\pi_q$.  
We define the following normalized period integral (i.e. Waldspurger integral)
\begin{equation}
\beta^0(v_q, v_q)
=\int\limits_{t\in \BQ_q^\times\backslash K_q^\times}\frac{ (\pi(t)v_q,v_q)}{(v_q,v_q)}\chi_q(t)dt.
\end{equation}

In this section, we will compute the ratio
\[\frac{\beta^0(\varphi'_q, \varphi'_q)}{\beta^0(\varphi_q,\varphi_q)}\]
for the new form $\varphi_q$ and the admissible test vector $\varphi'_q$
which appears in the proof of the explicit Gross-Zagier formulae. 

Since $p$ is split and $3$ is ramified in $K$, the local representation $\pi_p$ is a principal series with conductor $p$ while $\pi_3$ is supercuspidal with conductor $9$ or $27$ by \cite{GL}, we also say $\pi_p$ is of level $1$ and $\pi_3$ is of level $2$ or $3$. For any character $\xi$ of $\BZ_3^\times$ or $\CO_{K_3}^\times$, we say $\xi$ is of level $n$ if $\xi$ is trivial on $1+3^n\BZ_3$ or $1+\sqrt{-3}^n\CO_{K_3}$. The level of the representations and characters are denoted as $c(\pi_3)$ and $c(\xi)$. The computation splits into three cases and different cases require different methods.

\subsection{Local Waldspurger integral at prime $p$.}
Let $\Theta$ and $\chi$ be the unitary Hecke character as in Section \ref{EGZ}. For the split prime $p$, we write $\Theta_p$ and $\chi_p$ for the restriction of $\Theta$ and $\chi$ to $K_\varpi\oplus K_{\bar\varpi}$. Since $p$ is split in $K$, $\pi_p$ is the principal series $\pi(\Theta_\varpi,\Theta_{\bar\varpi})$ where $\Theta_\varpi=\Theta|_{K_\varpi},\Theta_{\bar\varpi}=\Theta|_{K_{\bar\varpi}}$. In the principal series case, the invariant Hermitian form on $\pi_p$ can be taken as

\[(v_1,v_2):=\int_{\BQ_p^\times}W_{v_1}\lrb{\matrixx{a}{}{}{1}}\ov{W_{v_2}\lrb{\matrixx{a}{}{}{1}}}d^\times a.\]
for any $v_1,v_2\in\pi_p$, where $W_{v_1}$ and $W_{v_2}$ are the Whittaker functions associated to $v_1$ and $v_2$. Recall that $\varphi\in\pi$ is the newform and its $p$-component $\varphi_p$ is the newform of $\pi_p$.

\begin{lemma}\label{trivial}
Let $p\equiv 1\mod 3$ be a prime and $p=\varpi\bar\varpi$ with $\varpi\equiv \pm 1\mod 3$ in $\CO_K$, then the cubic Hilbert symbol
 $$\lrb{\frac{\varpi,\bar\varpi}{K_3}}_3=\lrb{\frac{\varpi,\bar\varpi}{K_\varpi}}_3=\lrb{\frac{\varpi,\bar\varpi}{K_{\bar\varpi}}}_3=1.$$ 
\end{lemma}

\begin{proof}
By \cite[Exercise 7.18]{Reciprocity}, $\lrb{\frac{\varpi,\bar\varpi}{K_\varpi}}_3=\lrb{\frac{\varpi,\bar\varpi}{K_{\bar\varpi}}}_3=1$. By the product formula of Hilbert symbol, $\lrb{\frac{\varpi,\bar\varpi}{K_3}}_3=1$.

\end{proof}

For the construction of CM points in Section \ref{EGZ}, we use the embedding $\iota_1$ of $K$ into $M_2(\BQ)$ such that
\[\iota_1:\omega\mapsto \matrixx{-r}{-1/3}{3(r^2-r+1)}{r-1}.\]
For the split prime $p$, there is another simple embedding $\iota_3$ of $K\otimes\BQ_p$ into $M_2(\BQ_p)$ such that
\[\iota_3:\omega\mapsto \matrixx{\alpha^2}{0}{0}{\alpha},\]
here $\alpha$ is the unique cubic root of unity in $\BZ_p$ such that $p|(u+v\alpha)$ where $\varpi=u+v\omega$. Let $\iota_3\circ\iota_p^{-1}:K_\varpi\oplus K_{\bar\varpi}\cong K\otimes\BQ_p\hookrightarrow M_2(\BQ_p)$ be the composition of maps, then 
$$\iota_3\circ\iota_p^{-1}((a,1))=\matrixx{1}{0}{0}{a},\ \ \iota_3\circ\iota_p^{-1}((1,b))=\matrixx{b}{0}{0}{1},$$
for any $a\in K_\varpi$ and $b\in K_{\bar\varpi}$.
For more details on the above maps please see \cite[Section 4]{Yin26}.

The matrix
\[T'=\frac{-1}{3(2\alpha+1)}\matrixx{1}{1}{-3(r+\alpha^2)}{-3(r+\alpha)}\]
with
\[T'^{-1}=\matrixx{-3(r+\alpha)}{-1}{3(r+\alpha^2)}{1}\]
satisfies $$T'^{-1}\iota_1(\omega)T'=\matrixx{\alpha^2}{0}{0}{\alpha}.$$ Remember that we have chosen $r$ such that $r+\alpha^2\equiv 0\mod p$. Then $\pi_p(T'^{-1})\varphi_p=\varphi_p$.
So the Waldspurger integral of the new vector $\varphi_p$ is 
\begin{eqnarray*}
\int_{\iota_1(K_p^\times/\BQ_p^\times)}\frac{(\pi_p(t)\varphi_p,\bar{\varphi}_p)\chi_p(t)}{(\varphi_p,\bar{\varphi}_p)}dt
&=&\int_{\iota_3(K_p^\times/\BQ_p^\times)}\frac{(\pi_p(T'tT'^{-1})\varphi_p,\bar{\varphi}_p)\chi_p(t)}{(\varphi_p,\bar{\varphi}_p)_p}dt\\
&=& \int_{\iota_3(K_p^\times/\BQ_p^\times)}\frac{(\pi(t)\pi(T'^{-1})\varphi_p,\pi(T'^{-1})\bar{\varphi}_p)\chi_p(t)}{(\varphi_p,\bar{\varphi}_p)_p}dt\\
&=& \int_{\iota_3(K_p^\times/\BQ_p^\times)}\frac{(\pi(t)\varphi_p,\bar{\varphi}_p)\chi_p(t)}{(\varphi_p,\bar{\varphi}_p)_p}dt
\end{eqnarray*}
We choose the representatives $(1,b)$ of $K_p^\times/\BQ_p^\times$, the integral above becomes
\begin{eqnarray*}
&=& {(\varphi_p,\bar{\varphi}_p)_p}^{-1}\int_{b\in\BQ_p^\times}\int_{a\in\BQ_p^\times}W_{\varphi_p}\lrb{\matrixx{ab}{0}{0}{1}}\overline{W_{\varphi_p}\lrb{\matrixx{a}{0}{0}{1}}}\chi_p((1,b))d^\times a d^\times b\\
&=&{(\varphi_p,\bar{\varphi}_p)_p}^{-1}\int_{c\in\BQ_p^\times}W_{\varphi_p}\lrb{\matrixx{c}{0}{0}{1}}\chi_p((1,c))d^\times c\overline{\int_{a\in\BQ_p^\times}W_{\varphi_p}\lrb{\matrixx{a}{0}{0}{1}}\chi_p((1,a))d^\times a}
\end{eqnarray*}
Note that $\pi_p=\pi(\Theta_\varpi,\Theta_{\bar\varpi})$ where $\Theta_\varpi=\Theta|_{K_\varpi},\Theta_{\bar\varpi}=\Theta|_{K_{\bar\varpi}}$ and $\chi_p=(\chi_\varpi,\chi_{\bar\varpi})$ where $\chi_\varpi=\chi |_{K_\varpi},\chi_\varpi=\chi|_{K_{\bar\varpi}}$. By Proposition \ref{LCF} and Lemma \ref{trivial}, $\psi_\varpi,\chi_\varpi$ are of level $1$ and $\Theta_{\bar\varpi},\chi_{\bar\varpi}$ are both unramified characters. In fact $\chi_{\bar\varpi}$ is trivial while $\Theta_{\bar\varpi}(\bar\varpi)=\Theta(\bar\fp)=\bar\varpi/p^{1/2}$. By the description of the newforms in \cite[Page 23]{Schmidt},
\[W_{\varphi_p}\lrb{\matrixx{c}{0}{0}{1}}=\begin{cases}|c|_p^{1/2}\Theta_{\bar\varpi}(c),&\text{if}\ \ \ c\in\BZ_p-\{0\},\\ 0,& \text{otherwise}.\end{cases}\]
So, \[\int_{c\in\BQ_p^\times}W_{\varphi_p}\lrb{\matrixx{c}{0}{0}{1}}\chi_p((1,c))d^\times c=\frac{1}{1-\bar{\varpi} p^{-1}}.\]
Since $(\varphi_p,\bar{\varphi}_p)=L_p(1,\pi,\ad)$ by \cite[Proposition 3.11]{CST14}, we get that
\begin{eqnarray*}
\int_{\iota_1(K_p^\times/\BQ_p^\times)}\frac{(\pi_p(t)\varphi_p,\bar{\varphi}_p)\chi_p(t)}{(\varphi_p,\bar{\varphi}_p)}dt&=&\lrb{\frac{1}{1-\bar\varpi p^{-1}}}\lrb{\frac{1}{1-\varpi p^{-1}}}L_p(1,\pi,\ad)^{-1}\\
&=&\frac{L_p(E_1,1)}{L_p(1,\pi,\ad)}.
\end{eqnarray*}
Let $\varphi_p'$ be the admissible test vector in \cite{CST14}, then by \cite[Proposition 3.12]{CST14}, 
$$\beta^0(\varphi_p',\varphi_p')=\frac{L_p(1,1_\BQ)^2}{pL_p(1,\pi,\ad)}.$$
So we get
\begin{equation}\label{W1}
\frac{\beta^0(\varphi_p',\varphi_p')}{\beta^0(\varphi_p,\varphi_p)}=\frac{L_p(1,1_\BQ)^2}{pL_p(E_1,1)}=\frac{p}{(p-1)^2L_p(E_1,1)}.
\end{equation}
Note that the $L$-factor $L_p(E_1,1)$ completes the $L$ function in \eqref{G-Z}.

\subsection{Local Waldspurger integral at prime $3$ in the case $c(\pi_3)=2$}

For the newform $\varphi_3\in\pi_3$, define the local matrix coefficient
\[\Phi(t)=\frac{(\pi_3(t)\varphi_3,\varphi_3)}{(\varphi_3,\varphi_3)},\quad t\in \GL_2(\BQ_3).\]
Since $\chi_3$ has conductor $2$, we have
\begin{eqnarray}\label{beta}
\beta^0(\varphi_3,\varphi_3)&=&\frac{\Vol(K_3^\times/\BQ_3^\times)}
{ |K_3^\times/\BQ_3^\times (1+3\CO_{K,3})|}\sum_{t\in K_3^\times/\BQ_3^\times (1+3\CO_{K,3})}\Phi(\iota_1(t))\chi_{3}(t) \nonumber\\
&=&\frac{\Vol(K_3^\times/\BQ_3^\times)}
{ 6}\sum_{t\in S\bigsqcup S'}\Phi(\iota_1(t))\chi_{3}(t),
\end{eqnarray}
where
\[S=\{1+y\sqrt{-3}\mid y\in \BZ/3\BZ\},\quad S'=\{3y+\sqrt{-3}\mid y\in \BZ/3\BZ\}.\]
Note that $S\bigsqcup S'$ is a complete system of representatives of 
$$K_3^\times/\BQ_3^\times(1+3\CO_{K,3}).$$ 
In the above, we view $\chi$ as an adelic character through the class field theory and $\chi_3$ is the $3$-adic component of $\chi$. In order to compute $\beta_3^0(\varphi_3,\varphi_3)$, it suffices to compute the local matrix coefficients $\Phi(t)$ for $t\in S\bigsqcup S'$.

Let $\psi$ be the additive character such that $\psi(x)=e^{2\pi \sqrt{-1} \iota(x)}$ where $\iota:\BQ_3\rightarrow \BQ_3/\BZ_3 \subset \BQ/\BZ$ is the map given by $x\mapsto -x\mod \BZ_3$ and put $\psi^-(x)=\psi(-x)$. Let $dx$ be the Haar measure on $\BQ_3$ which is self-dual with respect to $\psi$, and we fix a Haar measure $d^\times x$ on $\BQ_3^\times$ such that $\Vol(\BZ_3^\times)=1$. The Kirillov model $\sK(\pi_3,\psi)$ is the unique realization of $\pi_3$ on the Schwartz function space $\CS(\BQ_3^\times)$ such that
$$\pi_3\left(\matrixx{a_1}{b}{0}{a_2}\right)\phi(x)=w_{\pi_3}(a_2)\psi( bx/a_2)\phi(a_1x/a_2),\quad \forall\ \phi\in \CS(\BQ_3^\times).$$
The $\GL_2(\BQ_3)$-invariant pairing $(\cdot,\cdot)$ on $\pi_3\times \pi_3$ is given by
\[(\phi_1,\phi_2)=\int_{\BQ_3^\times}\phi_1(x)\ov{\phi_2(x)}d^\times x.\]
Put
$$1_{\nu,n}(x)=\begin{cases}\nu(u),& \textrm{if $ x=u3^{n}$ for $ u\in\BZ_3^{\times}$},\\
                             0,& \mathrm{otherwise},\end{cases}$$
where $\nu$ is a character of $\BZ_3^{\times}$.
Then $\{1_{\nu,n}(x)\}_{\nu,n}$ is an orthogonal basis of $\CS(\BQ_3^\times)$ with respect to the paring $(\cdot,\cdot)$. For $\phi(x)\in\CS(\BQ_3^\times)$, we have the Fourier expansion
$$\phi(x)=\sum_{\nu}\sum_n\widehat{\phi}_n(\nu^{-1})1_{\nu,n},$$
where
$$\widehat{\phi}_n(\nu^{-1})=\int_{\BZ_3^\times}\phi(3^nx)\nu^{-1}(x)d^\times x.$$
The action of $S=\matrixx{0}{1}{-1}{0}$ on $1_{\nu,n}$ can be described as follows:
$$\pi_3\left(\matrixx{0}{1}{-1}{0}\right)1_{\nu,n}=C_{\nu w_0^{-1}}z_0^{-n} 1_{\nu^{-1}w_0,-n+n_{\nu^{-1}}}.$$
Here $w_0=w_{\pi_3}|_{\BZ_3^\times}=\chi_3^{-1}\mid_{\BZ_3^\times}$, $z_0=w_{\pi_3}(3)=1$,
$$C_{\nu}=\epsilon(1/2,\pi_3\otimes \nu^{-1}, \psi),$$
and $n_{\nu}=-\max\{c(\pi_3), 2c(\nu)\}$, where $c(\nu)$ is the conductor of $\nu$ and $c(\pi_3)=2$ or $3$ is the conductor of $\pi_3$, see \cite[Proposition 2.15]{Hu17}. From $S^2=-1$, one can see that
$$n_{\nu}=n_{\nu^{-1}w_0^{-1}},\ \ \ C_{\nu}C_{\nu^{-1}w_0^{-1}}=w_0(-1)z_0^{n_\nu}.$$
For the basics of supercuspidal representations, the readers may refer to \cite{Saito1993}. It is well-known that $1_{1,0}$ is the normalized local new form, and hence, is parallel to $\varphi_3$.

We will use 
the following decomposition of matrices:
\begin{equation}\label{decomposition1}
\matrixx{a}{3^jb}{3^kc}{d}=\matrixx{ac^{-1}-bd^{-1}3^{k+j}}{d^{-1}b3^j}{0}{1}\matrixx{1}{0}{3^k}{1}\matrixx{c}{0}{0}{d}.
\end{equation}
\begin{equation}\matrixx{1}{0}{3^k}{1}=-S\matrixx{1}{-3^k}{0}{1}S.
\end{equation}
Under the embedding $\iota_1$, we have
\begin{eqnarray}\label{matrix}
\iota_1(\sqrt{-3})&=&\matrixx{1-2r}{-2/3}{9\cdot\frac{2(r^2-r+1)}{3}}{2r-1} \nonumber\\
&=&\matrixx{2r-1}{0}{0}{2r-1}\matrixx{\frac{3}{2(r^2-r+1)(2r-1)}}{\frac{-2}{3(2r-1)}}{0}{1}\matrixx{1}{0}{3}{1}\matrixx{\frac{2(r^2-r+1)}{2r-1}}{0}{0}{1}\nonumber\\
&&
\end{eqnarray}
we also have
$$\iota_1(1+y\sqrt{-3})=\matrixx{1+(1-2r)y}{\frac{-2}{3}y}{9\cdot\frac{2(r^2-r+1)}{3}y}{1-(1-2r)y},\quad y\in \BZ/3\BZ.$$
Assume $y\neq 0$, we have the decomposition
\begin{equation}\label{matrix1}
\iota_1(1+y\sqrt{-3})=\matrixx{A_y}{\frac{B_y}{3}}{0}{1}\matrixx{1}{0}{3^2}{1}\matrixx{\frac{2(r^2-r+1)y}{3}}{0}{0}{(2r-1)y+1}
\end{equation}
with 
$$A_y=\frac{3(1+3y^2)}{2y(r^2-r+1)((2r-1)y+1)},\ \ \ B_y=\frac{-2y}{(2r-1)y+1}$$ 
both 3-units. Here we use the fact that 
$\ord_3(r^2-r+1)=1$ since $r\equiv 2\mod 3$ by our condition. 
Moreover if $r\equiv 2\mod 3$, then $r^2-r+1\equiv 3\mod 9$ and $2r-1\equiv 0\mod 3$. As a result, 
\begin{equation}\label{AB}
A_y\equiv 2y\mod 3,\ \ \ B_y\equiv y\mod 3.
\end{equation}
To simplify the computations, we also choose $r$ such that 
$$\ord_3(2r-1)=1.$$ 
This choice will not affect the result since all the $r$ which satisfies the conditions in Theorem \ref{non-torsion} will give the same $K$ points in fact. This can be seen from the fact that $\varphi^c(W(\tau_r))=\varphi^c\lrb{\frac{3(\omega+r)}{N}}$ is independent of the choice of $r$ and $\phi^c\circ\varphi^c(W(\tau_r))=\phi\circ\varphi(\tau_r)$ up to a torsion point by \cite[(8.0.7)]{Yin26}, here $W$ is the Atkin-Lehner matrix.

If $y\neq 0$, then by (\ref{matrix1})

\begin{equation*}
\pi_3(1+y\sqrt{-3})1_{1,0}(x)=\psi\lrb{B_y\frac{x}{3}}1_{1,0}(x)
\end{equation*}
and

\begin{eqnarray*}
\Phi(1+y\sqrt{-3})&=&(\pi_3(1+y\sqrt{-3})1_{1,0},1_{1,0})\\
&=&\int_{\BZ_3^\times}\psi\left(\frac{B_yx}{3}\right)d^\times x\nonumber\\
&=&-\frac{1}{2}.
\end{eqnarray*}
So, 
\begin{equation}\label{mc}
\Phi(1+y\sqrt{-3})=\begin{cases}1,& y=0;\\ -\frac{1}{2},& y=1,2.\end{cases}
\end{equation}

Next, we compute $\Phi(3y+\sqrt{-3})$. First of all, we have

\begin{eqnarray*}
\pi_3\lrb{\matrixx{1}{-3}{0}{1}S}1_{1,0}(x)&=&C_{w_0^{-1}}\psi(-3x)1_{w_0,-2}(x)\\
&=&C_{w_0^{-1}}\lrb{-\frac{1}{2}1_{w_0,-2}(x)+\frac{\sqrt{-3}}{2}1_{\nu_1w_0,-2}(x)}
\end{eqnarray*}
where $\nu_1$ is the unique character of $\BZ_3^\times$ with level $1$ (in fact $\nu_1=w_0$) and in the last equation we use the Fourier expansion of $\psi(-3x)$. Apply another action of $S$, we get
\[\pi_3\lrb{\matrixx{1}{0}{3}{1}}1_{1,0}(x)=-\frac{w_0(-1)}{2}1_{1,0}(x)+\frac{\sqrt{-3}}{2}C_{w_0^{-1}}C_{\nu_1} 1_{\nu_1,0}(x).\]
Then using the decomposition (\ref{matrix}), we see 
\begin{eqnarray*}
\frac{\pi_3(\sqrt{-3})1_{1,0}(x)}{w_0(2r-1)}&=&-\frac{w_0(-1)}{2}\psi\lrb{\frac{-2x}{3(2r-1)}}1_{1,1}(x)+\\
&&\frac{\sqrt{-3}}{2}C_{w_0^{-1}}C_{\nu_1}\nu_1\lrb{\frac{9}{2(r^2-r+1)(2r-1)}}\psi\lrb{\frac{-2x}{3(2r-1)}}1_{\nu_1,1}(x).
\end{eqnarray*}
As a result, we know that for all $y$
\begin{eqnarray}\label{mc1}
\Phi(3y+\sqrt{-3})&=&(\pi_3(\sqrt{-3}(1-y\sqrt{-3}))1_{1,0},1_{1,0})\\
&=&(\pi_3(-3)\pi(1-y\sqrt{-3})1_{1,0},\pi_3(\sqrt{-3})1_{1,0}) \nonumber\\
&=&0. \nonumber
\end{eqnarray}
Combining (\ref{beta}), (\ref{mc}) and (\ref{mc1}), we get
\begin{equation}\label{beta1}
\beta^0(\varphi_3,\varphi_3)=\frac{\Vol(K_3^\times/\BQ_3^\times)}
{6}\lrb{1-\frac{1}{2}(\omega+\omega^2)}=\frac{\Vol(K_3^\times/\BQ_3^\times)}
{4}.
\end{equation}

\begin{proposition}\label{TestingForNew}
If $c(\pi_3)=2$, for $\varphi_3$ being the newform corresponding to $\pi_3$ and $K$ being embedded in $M_2(\BQ)$ as in $(\ref{emb})$, we have
\begin{equation*}
\frac{\beta^0(\varphi'_3,\varphi'_3)}{\beta^0(\varphi_3,\varphi_3)}=4,
\end{equation*}
where $\varphi'_3$ is the admissible test vector.
\end{proposition}
\begin{proof}
By the definition of the admissible test vector, 
\[\beta^0\lrb{\varphi'_3,\varphi'_3 }=\Vol(\BQ_3^\times\bs K_3^\times).\]
Then the result is clear from (\ref{beta1}).
\end{proof}

\subsection{Local Waldspurger integral at prime $3$ in the case $c(\pi_3)=3$}

In this section we have $c(\pi_3)=3$. By (\ref{conductor}),

\[\Theta(\fq)=\begin{cases}\ov{\lrb{\frac{\varpi^2}{\fq}}_3}\frac{\varpi_\fp}{\sqrt{N\fp}}& p\equiv 4\mod 9,\\
\ov{\lrb{\frac{\varpi}{\fq}}_3}\frac{\varpi_\fp}{\sqrt{N\fp}},& p\equiv 7\mod 9,\end{cases}\]
where $\varpi_\fq$ is a generator of $\fq$ such that $\varpi_\fq\equiv 1\mod 3$ for $\fq\nmid 3\varpi$.

\begin{lemma}\label{character}
Let $\varepsilon\in \CO_\varpi^\times$, then 
$$\Theta_\varpi(\varepsilon)=\begin{cases}\lrb{\frac{\varepsilon}{\varpi}}_3^2,& p\equiv 4\mod 9,\\
\lrb{\frac{\varepsilon}{\varpi}}_3,& p\equiv 7\mod 9.\end{cases}$$
\end{lemma}
\begin{proof}Let $\kappa:J^{3\varpi}\ra C(3\varpi)$ be the homomorphism from the ideal class group of $K$ to the idele class group of $K$ which maps a prime ideal $\fq \nmid 3\varpi$ to the class of the idele $(\cdots,1,1,\pi_\fq,1,1,\cdots)$ as in \cite[P481]{Neukirchbook1}. Let $k\in \BN$ such that $k\equiv 1\mod 3$ and $k\equiv \varepsilon\mod \varpi$. Then $$\kappa((k))^{-1}k_\infty^{-1}= \varepsilon_\varpi \mod K^\times U_{3\varpi}$$ where $\varepsilon_\varpi$ is the idele with $\varepsilon$ at place $\varpi$ and $1$ at other places. So if $p\equiv 7\mod 9$, $\Theta_\varpi(\varepsilon)=\Theta(\kappa(k))^{-1}\Theta_\infty(k)^{-1}=\Theta(\kappa(k))^{-1}=\ov{\lrb{\frac{\varpi}{k}}}_3^{-1}=\lrb{\frac{k}{\varpi}}_3=\lrb{\frac{\varepsilon}{\varpi}}_3$ where we use the fact that $k\equiv 1\mod 3$ and the cubic reciprocity law. Similarly, we get the result for the case $p\equiv 4\mod 9$.
\end{proof}

\begin{lem}\label{character1}
We have \[\CO_{K,3}^\times/(1+3\CO_{K,3})=\langle -1\rangle^{\BZ/2\BZ}\times \langle 1+\sqrt{-3} \rangle^{\BZ/3\BZ},\]

\[\Theta_3(-1)=-1,\ \Theta_3(1+\sqrt{-3})=\omega,\ \Theta_3(\sqrt{-3})=\begin{cases}\sqrt{-1}\cdot\lrb{\frac{\sqrt{-3}}{\varpi}}_3,& p\equiv 4\mod 9,\\
\sqrt{-1}\cdot\ov{\lrb{\frac{\sqrt{-3}}{\varpi}}}_3,& p\equiv 7\mod 9.\end{cases}\]
\end{lem}

\begin{proof}It is well-known that $\Theta_\infty(x)=\frac{||x||}{x}$, by Lemma \ref{character},
$$\Theta_3(-1)=(\Theta_\varpi(-1)\Theta_\infty(-1))^{-1}=-1.$$

By Lemma \ref{character}, $\Theta_\varpi(\omega)=\omega^2$. Since $\Theta_2$ is unramified and $\Theta_3(-2)=\Theta_3(1)=1$, we have
\begin{eqnarray*}
\Theta_3(1+\sqrt{-3})&=&\Theta_\infty(1+\sqrt{-3})^{-1}\Theta_2(1+\sqrt{-3})^{-1}\Theta_\varpi(1+\sqrt{-3})^{-1}\\
&=& -\omega^2\Theta_\varpi(\omega^2)^{-1}\Theta_2(-2)^{-1}\Theta_{\varpi}(-2)^{-1}\\
&=&\omega^2\Theta_\varpi(\omega)\Theta_3(-2)\\
&=&\omega
\end{eqnarray*}
Finally, 
$$\Theta_3(\sqrt{-3})=(\Theta_\infty(\sqrt{-3})\Theta_{\varpi}(\sqrt{-3}))^{-1}=\begin{cases}\sqrt{-1}\cdot\lrb{\frac{\sqrt{-3}}{\varpi}}_3,& p\equiv 4\mod 9,\\
\sqrt{-1}\cdot\ov{\lrb{\frac{\sqrt{-3}}{\varpi}}}_3,& p\equiv 7\mod 9.\end{cases}$$
\end{proof}

\begin{lem} \label{chi}
We have $c(\chi_3)=2$ and $\chi_3$ is given explicitly as follows:
\[\chi_3(-1)=1,\ \chi_3(1+\sqrt{-3})=\omega,\ \chi_3(\sqrt{-3})=\begin{cases}\lrb{\frac{\sqrt{-3}}{\varpi}}_3^2,& p\equiv 4\mod 9,\\
\lrb{\frac{\sqrt{-3}}{\varpi}}_3,& p\equiv 7\mod 9.\end{cases}\]
\end{lem}

Let $\chi_3^c$ be the character on $K_3$ such that $\chi_3^c(x)=\chi_3(\bar{x})$.

\begin{coro}
\label{thetachivalue}
The local character $\Theta_3\chi_3^c$ has level one and is given explicitly by
\[\Theta_3\chi_3^c(-1)=-1,\ \quad \Theta_3\chi_3^c(1+\sqrt{-3})=1,\ \quad \Theta_3\chi_3^c(\sqrt{-3})=\sqrt{-1}.\]
\end{coro}

Let $\theta_3$ be the $3$-adic character which parametrizes the supercuspidal representation $\pi_3$ via the compact-induction construction, then $c(\theta_3)=2$ see \cite[section 2.2]{HSY2}. The test vector issue for Waldspurger's local period integral is closely related to $c(\theta_3\chi_3^c)$ or $c(\theta_3\chi_3)$. We can work out these  by using Lemma \ref{character1}, \ref{chi} and Corollary \ref{thetachivalue}, and the relation between $\theta_3$ and $\Theta_3$ in \cite[Theorem 2.10]{HSY2}.

Let $\psi$ be the additive character such that $\psi(x)=e^{2\pi \sqrt{-1} \iota(x)}$ where $\iota:\BQ_3\rightarrow \BQ_3/\BZ_3 \subset \BQ/\BZ$ is the map given by $x\mapsto -x\mod \BZ_3$. 
For any additive character $\xi$ of $\BQ_3$,  the Langlands $\lambda$-function of the extension $K_3/\BQ_3$ is $\lambda_{K_3/\BQ_3}(\xi):=\epsilon(\Ind_{G_{\BQ_3}}^{G_{K_3}}(1_{K_3}), \xi)$ where $1_{K_3}$ is trivial representation of the absolute Galois group $G_{K_3}$ and $\epsilon(\cdot,\cdot)$ means the local $\epsilon$-factor, see \cite{Langlands}. 
Now we prove the following key lemma. 
\begin{lem}\label{Cor:AllnecessaryFormulationForThetaChi}
We have $\theta_3\chi_3^c$ is the trivial character.
\end{lem}
\begin{proof}
Let $\psi_{K_3}(x)=\psi\circ \Tr_{K_3/\BQ_3}(x)$, be the additive character of $K_3$. Then $\alpha_{\Theta_3}=1/(3\sqrt{-3})$ satisfies $\Theta_3(1+x)=\psi_{K_3}(\alpha_{\Theta_3}x)$ for any $x$ satisfying $\ord_{\sqrt{-3}}(x)\geq 1$. Now let $\eta$ be the quadratic character associated with the quadratic field extension $K_3/\BQ_3$.
Then by \cite[Proposition 34.3]{BushnellHenniart:06a}, $$\lambda_{K_3/\BQ_3}(\psi')=\tau(\eta,\psi')/\sqrt{3}=-i,$$ here $\tau(\eta,\psi')$ is the Gauss sum and $\psi'(x)=\psi(\frac{x}{3})$ is the additive character of level one. By \cite[Lemma 5.1]{Langlands}, $$\lambda_{K_3/\BQ_3}(\psi)=\eta(3)\lambda_{K_3/\BQ_3}(\psi')=-\sqrt{-1}.$$ 
Then define $\Delta_{\theta_3}$ to be the unique level one character of $K_3$ such that $\Delta_{\theta_3}|_{\BZ_3^\times}=\eta$ and
\[\Delta_{\theta_3}(\sqrt{-3})=\eta((\sqrt{-3})\alpha_{\Theta_3})\lambda_{K_3/\BQ_3}(\psi)=-\sqrt{-1}.\]
Then by \cite[Theorem 2.10]{HSY2}, $\theta_3=\Theta_3\Delta_{\theta_3}$. By Corollary \ref{thetachivalue} we can easily know that $\theta_3\chi_3^c$ is trivial.
\end{proof}

In our case, $c(\theta_3)=c(\Theta_3)=c(\chi_3)=2$, let $n=(c(\pi_3)-1)/2=1$. For the supercuspidal representation $\pi_3$ of $\GL_2(\BQ_3)$, 
the Kirillov model $\sK(\pi_3,\psi)$ is the unique realization of $\pi_3$ on the Schwartz function space $\CS(\BQ_3^\times)$ such that
\begin{equation}\label{kiri}
\pi_3\left(\matrixx{a}{b}{0}{1}\right)\varphi(y)=\psi( by)\varphi(ay),\quad \varphi\in \CS(\BQ_3^\times).
\end{equation}
By \cite[Lemma 2.11]{HSY2}, we have the minimal vector $\varphi_0=\Char(3^{-2}(1+\BZ_3^\times))$ in the Kirillov model. For more details on minimal vectors, we refer to \cite[2.5]{HSY2}.  Recall under the embedding $\iota_1$, $K$ is embedded into $\rm{M}_2(\BQ)$ such that:
\begin{equation}\label{emb}
\sqrt{-3}\mapsto \matrixx{a}{3^{-1}b}{3^2c}{-a}:=\matrixx{1-2r}{-2/3}{9\cdot\frac{2(r^2-r+1)}{3}}{2r-1}
\end{equation}
where $b, c\in \BQ\cap \BZ_3^\times$.

\begin{prop}\label{Prop:TestingForNew}
If $c(\pi_3)=3$, for $\varphi_3$ being the newform corresponding to $\pi_3$ and $K$ being embedded in $M_2(\BQ)$ as in $(\ref{emb})$, we have
\begin{equation*}
\frac{\beta^0(\varphi'_3,\varphi'_3)}{\beta^0(\varphi_3,\varphi_3)}=2,
\end{equation*}
where $\varphi'_3$ is the admissible test vector.
\end{prop}

\begin{proof}
We may assume $\varphi_3$ to be $L^2$-normalized and use the notation
\begin{equation}
\BetaI{\varphi,\varphi'}:=\int\limits_{t\in \BQ_3^\times\backslash K_3^\times}(\pi(t)\varphi,\varphi')\chi_3(t)dt
\end{equation}
for the usual local integral. So under our choice of $\varphi_3$, $\beta^0_3(\varphi_3,\varphi_3)=\BetaI{\varphi_3,\varphi_3}$.

To evaluate $\varphi_3$ for the embedding in \eqref{emb} is equivalent to using the standard embedding 
\begin{equation}\label{standardemb}\sqrt{-3}\mapsto \matrixx{0}{1}{-3}{0}\end{equation}
of $K_3$ for the corresponding translation of the newform.
In particular, the embedding in \eqref{emb} is conjugate to the standard embedding by the following 
\begin{equation}
    \matrixx{a}{3^{-1}b}{3^2c}{-a}=\matrixx{-3c}{a/3}{0}{1}^{-1}\matrixx{0}{1}{-3}{0}\matrixx{-3c}{a/3}{0}{1},
\end{equation}
where we have used the fact that $\mathrm{Nm}(\sqrt{-3})=-a^2-3bc=3$.
Thus  we have
\begin{align}
\beta^0(\varphi_3,\varphi_3)&=\int\limits_{\BQ_3^\times\backslash K_3^\times}\left(\pi_3\lrb{\matrixx{-3c}{a/3}{0}{1}^{-1}t\matrixx{-3}{a/3}{0}{1}}\varphi_3,\varphi_3\right)\chi_3(t)dt\\
&=\int\limits_{\BQ_3^\times\backslash K_3^\times}\left(\pi_3\lrb{t\matrixx{-3c}{a/3}{0}{1}}\varphi_3,\pi_3\lrb{\matrixx{-3c}{a/3}{0}{1}\varphi_3}\right)\chi_3(t)dt,\label{standardintegral}
\end{align}
where $K_3$ is embedded in $M_2(\BQ_3)$ as in (\ref{standardemb}). By definition, the integral in (\ref{standardintegral}) is just
$$\BetaI{\pi_3\lrb{\matrixx{-3c}{a/3}{0}{1}}\varphi_3,\pi_3\lrb{\matrixx{-3c}{a/3}{0}{1}}\varphi_3}$$
for the standard embedding.
Note that by \cite[Corollary 2.12]{HSY2},
\begin{equation}\label{new-min}
\pi_3\lrb{\matrixx{-3c}{a/3}{0}{1}}\varphi_3
=\frac{1}{\sqrt{(q-1)}}\sum\limits_{x\in (\BZ_3/3\BZ_3)^\times}\pi_3\lrb{\matrixx{1}{a/3}{0}{1}\matrixx{x}{0}{0}{1}}\varphi_0
\end{equation}
where $\lceil\cdot\rceil$ means the smallest integer not less than the given number and $e=2$ is the ramification index of $K_3/\BQ_3$. Denote 
$$\varphi_{a, x}=\pi_3\lrb{\matrixx{1}{a/3}{0}{1}\matrixx{x}{0}{0}{1}}\varphi_0.$$
In order to compute $\beta^0_3(\varphi_3,\varphi_3)$, we just need to consider $\{\varphi_{a, x'},\varphi_{a, x''}\}$ for $x',x''\in (\BZ_3/3\BZ_3)^\times$.

By (\ref{kiri}), $\varphi_{a, x}(y)=\psi((a/3)y)\varphi_0(xy)$.
Then $\{\varphi_{a, x'},\varphi_{a, x''}\}=\mu\{\varphi_{0, x'},\varphi_{0, x''}\}$ for some roots of unity $\mu$. If $x'=x''$, $\mu=1$ since we take a dual pair. By the $c(\theta_3\chi_3^c)=0$ case in \cite[Section 2.4]{HSY2}, we have a unique $x\mod 3$ for which $\BetaI{\varphi_{0,x},\varphi_{0,x}}=\Vol(\BQ_3^\times\bs K_3^\times)$ is nonvanishing.  By Proposition \cite[Proposition A.4]{Yin1}, we have
\begin{equation}
\beta^0\lrb{\varphi_3,\varphi_3 }=\frac{1}{(q-1)q^{\lceil \frac{c(\theta_3)}{2 e_\BL}\rceil-1}} \BetaI{\varphi_{0,x},\varphi_{0,x}}=\frac{1}{2}\Vol(\BQ_3^\times\bs K_3^\times).
\end{equation}
By the definition of the admissible test vector, 
\[\beta^0\lrb{\varphi'_3,\varphi'_3 }=\Vol(\BQ_3^\times\bs K_3^\times).\]
Then the result is clear.
\end{proof}

\bibliographystyle{alpha}
\bibliography{reference}
\end{document}